\documentclass[a4paper,USenglish,cleveref,autoref,thm-restate,pdfa]{lipics-v2021}
\hideLIPIcs\def\subjclassHeading{}
\nolinenumbers 
\allowdisplaybreaks[1]

\title{Hunting for Directed 2-Spiders}

\author{Grzegorz Gutowski}
  {Institute of Theoretical Computer Science, Faculty of Mathematics and Computer Science, Jagiellonian University, Krak{\'o}w, Poland}
  {grzegorz.gutowski@uj.edu.pl}
  {https://orcid.org/0000-0003-3313-1237}
  {Partially supported by grant no.~2023/49/B/ST6/01738 from National Science Centre, Poland.}
\author{Gaurav Kucheriya}
  {Department of Applied Mathematics, Charles University, Prague, Czechia}
  {gaurav@kam.mff.cuni.cz}
  {https://orcid.org/0000-0002-5254-5198}
  {Partially supported by GA\v{C}R grant 25-17377S and GAUK project 92125.}

\authorrunning{G.~Gutowski, G.~Kucheriya}
\Copyright{Grzegorz Gutowski, Gaurav Kucheriya}


\acknowledgements{
  The authors would like to thank the organizers of the 10th Cracow Conference
  on Graph Theory (2025) which brought the authors together to work on this problem.
  The authors are also grateful to the organizers of Homonolo 2025, where the final stages of this research was carried out. The authors also thank Mykhaylo Tyomkyn for carefully reading an earlier draft of this paper and for helpful comments. 
}

\usepackage{amsmath,amsfonts,amssymb,amsthm}
\usepackage{esvect}
\usepackage{graphicx}
\usepackage{stackengine}
\usepackage{tikz}
\usepackage[textsize=small]{todonotes}
\usepackage{xcolor}
\usepackage{xspace}
\usepackage{mathtools}

\let\geq\geqslant
\let\le\leqslant
\let\ge\geqslant

\let\rho\varrho

\newcommand{\set}[1]{{\left\{#1\right\}}}
\newcommand{\norm}[1]{{\left|#1\right|}}

\newcommand{\brac}[1]{\left(#1\right)}
\newcommand{\Oh}[1]{O\brac{#1}}
\newcommand{\oh}[1]{o\brac{#1}}
\newcommand{\noto}{\not\to}

\newcommand{\pathset}[1]{\vv{#1}}

\definecolor{dark blue}{rgb}{0.121,0.47,0.705}
\let\emph\relax\DeclareTextFontCommand{\emph}{\color{dark blue}\em}

\begin{document}

\maketitle

\begin{abstract}
    Hons, Klimošová, Kucheriya, Mikšaník, Tkadlec, and Tyomkyn proved that, for every integer $\ell \ge 1$, every directed graph with minimum out-degree at least $3.23 \cdot \ell$ contains a $(2,\ell)$\nobreakdash-spider (a $1$\nobreakdash-subdivision of the in-star with $\ell$ leaves) as a subgraph.
    They also conjectured that the bound on the minimum out-degree can be further improved to $2 \ell$.
    In this note, we confirm their conjecture by showing that
    every directed graph with minimum out-degree at least $2\ell$ contains a $(2, \ell)$\nobreakdash-spider as a subgraph.
    This result is best possible, as the complete directed graph with $2\ell$ vertices does not contain a $(2,\ell)$\nobreakdash-spider.
\end{abstract}

\section{Introduction}

The question of how local degree conditions force global structure has played a central role in the development of extremal graph theory.
A classical result of Dirac~\cite{Dirac1952} shows that every graph with $n \ge 3$ vertices and minimum degree at least $n/2$ contains a Hamiltonian cycle.
Ore~\cite{Ore1960} strengthened Dirac’s theorem by proving that a graph with $n$ vertices is Hamiltonian whenever the sum of the degrees of any two non-adjacent vertices is at least $n$.
In a different direction, Mader proved that for every integer $k$, there is an integer $d_k$ such that every finite graph with minimum degree $d_k$ contains a $k$-connected subgraph~\cite{Mader1972a} and a subdivision of the complete graph $K_k$~\cite{Mader1972b}.

In directed graphs, this line of research becomes more subtle.
Already at the level of degree conditions, there is no longer a unique notion of vertex degree, and one may impose constraints either on both the in-degrees and out-degrees (so-called semidegree conditions) or on only one of these parameters.
As a consequence, the extent to which local degree assumptions force global structural properties in digraphs is far less clear than in the undirected setting.
For instance, in \cite{Mader1985}, Mader presents a construction which shows that even imposing large semidegree conditions does not always guarantee the existence of highly connected subgraphs in digraphs.

When only minimum out-degree is considered, the directed setting presents additional challenges and several longstanding open problems. A prominent example is the Caccetta--H{\"a}ggkvist conjecture~\cite{caccetta1978minimal}, which asserts that every directed graph with $n$ vertices with minimum out-degree at least $n/k$ contains a directed cycle of length at most $k$.
Partial progress has been made on the Caccetta--H{\"a}ggkvist conjecture in special cases: Caccetta--H{\"a}ggkvist~\cite{caccetta1978minimal} verified the case $k=2$, whereas the case $k=3$ was shown by Hamidoune~\cite{Hamidoune1987}.

Recently, Hons, Klimošová, Kucheriya, Mikšaník, Tkadlec, Tyomkyn~\cite{hons2025unavoidable} initiated a systematic study of the following natural question arising from minimum out-degree conditions: given a directed graph $G$ with sufficiently large minimum out-degree, which directed graphs $H$ are \emph{unavoidable}, i.e., which directed graphs $H$ must be contained in $G$ as a subgraph?
Aboulker, Cohen, Havet, Lochet, Moura and Thomassé~\cite{aboulker1610subdivisions} proved that every in-arborescence $T$
    (i.e.\ a tree with all edges oriented towards a designated root vertex $r$)
    is unavoidable in every directed graph $G$ with minimum out-degree $d_T$ for some $d_T$ depending on the in-arborescence $T$.

An oriented graph is a \emph{$(k,\ell)$\nobreakdash-spider} if it is the $(k-1)$-subdivision of the in-star with $\ell$ leaves.
We say that a $(k,\ell)$\nobreakdash-spider is \emph{rooted} at a vertex~$r$ when $r$ is the central vertex of the subdivided in-star.
See \cref{fig:spider} for an example. 

\begin{figure}[ht]
    \centering
    \begin{tikzpicture}[>=stealth, every node/.style={circle, draw, fill=black, inner sep=1.5pt}]

    \node[label=below:$r$] (C) at (0,0) {};

    \coordinate (L1) at (2,1);
    \coordinate (L2) at (2,-1);
    \coordinate (L3) at (-2,1);
    \coordinate (L4) at (-2,-1);

    \node (S1) at (1,0.5) {};
    \node (S2) at (1,-0.5) {};
    \node (S3) at (-1,0.5) {};
    \node (S4) at (-1,-0.5) {};

    \draw[->] (L1) -- (S1);
    \draw[->] (S1) -- (C);

    \draw[->] (L2) -- (S2);
    \draw[->] (S2) -- (C);

    \draw[->] (L3) -- (S3);
    \draw[->] (S3) -- (C);

    \draw[->] (L4) -- (S4);
    \draw[->] (S4) -- (C);

    \foreach \L in {L1,L2,L3,L4} {
        \node at (\L) {};
    }
\end{tikzpicture}
    \caption{A $(2,4)$\nobreakdash-spider rooted at $r$.}
    \label{fig:spider}
\end{figure}
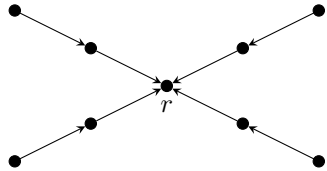

A $(k,\ell)$\nobreakdash-spider is an in-arborescence, and thus it is unavoidable. The next natural question is to determine the minimum out-degree $d$ that any directed graph $G$ must have so that a $(k,\ell)$\nobreakdash-spider is unavoidable. In this direction, generalizing a conjecture for oriented graphs of Thomassé~\cite{bang2008digraphs,sullivan2006summary}, Hons et al.\ proposed the following Giant Spider Conjecture:
 
\begin{conjecture}[Giant Spider Conjecture~\cite{hons2025unavoidable}]\label{conj:kl_spider}
    For all integers $k \ge 2$, $\ell \ge 1$, every directed graph with minimum out-degree at least $k\ell$ and every oriented graph with minimum out-degree at least $k\ell / 2$ contains a $(k,\ell)$\nobreakdash-spider as a subgraph.
\end{conjecture}

The bounds conjectured above would be tight as witnessed by the examples of the complete directed graph and a regular tournament of order $k\ell$, respectively. 
In~\cite{hons2025unavoidable}, the authors proved a relaxed version of \cref{conj:kl_spider} for $k=2$, directed graphs with a higher multiplicative constant.
They showed that every directed graph with minimum out-degree at least $3.23\cdot\ell$ contains a $(2,\ell)$\nobreakdash-spider as a subgraph.
However, \cref{conj:kl_spider} is still open.
In this note, we confirm \cref{conj:kl_spider} for $k=2$, every $\ell \ge 1$ and every directed graph by proving the following theorem.
\begin{theorem}\label{thm:main}
    For every integer $\ell \ge 1$, every directed graph with minimum out-degree at least $2\ell$ contains a $(2, \ell)$\nobreakdash-spider as a subgraph.
\end{theorem}

For a directed graph $G$, let $X,Y,Z \subseteq V(G)$ be three subsets of vertices.
We denote by $\pathset{XY}$ the set of all edges oriented from $X$ to $Y$.
We denote by $\pathset{XYZ}$ the set of all simple $2$-paths oriented from $X$ to $Y$ to $Z$.
For a vertex $v \in V(G)$, we denote by $N^+(v)$ the set of out-neighbors of $v$ in $G$, and by $N^-(v)$ the set of in-neighbors of $v$ in $G$.
For a vertex $v$, we use $\delta^+(v) = \norm{N^+(v)}$ and $\delta^-(v) = \norm{N^-(v)}$ to denote the out-degree of $v$ and the in-degree of $v$, respectively. 
We write $x \to y$ to denote that there is an edge $(x,y) \in E(G)$ and $x \noto y$ to denote that $(x,y)$ is not an edge in $G$.
We write $x \to y \to z$ to denote $(x \to y) \land (y \to z) \land (x \neq z)$.

\section{Main Result}
Let $G$ be a directed graph, and $r$ be any vertex in $V(G)$.
For any other vertex $x \in V(G)\setminus\set{r}$, let 
\[
    O_{x,r} \coloneqq \set{ y \in V(G) : (x \to y \to r) \lor (y \to x \to r)}
\]
denote the set of vertices that together with $x$ form a simple $2$-path that ends in $r$.
For any integer $i \ge 1$, we say that a vertex $x$ is an \emph{$i$-extender} for $r$ if and only if $\norm{O_{x,r}} \ge i$.
Intuitively, when $r$ is a fixed root of a spider and $x$ is an $i$-extender for $r$, then there are at least $i$ different vertices that can attach $x$ to a spider rooted at $r$.
Before we proceed to the proof of \cref{thm:main}, we prove the following technical lemma that allows us to use extenders to extend the size of a constructed spider.
\begin{lemma}\label{lem:greedy}
    For all integers $f \ge 0$, $s \ge 0$, a directed graph $G$, and a vertex $r \in V(G)$, if $G$ contains a set $F$ of $f$ vertices, $F = \set{x_1,x_2,\ldots,x_f}$, such that for each $1\le i \le f$, vertex $x_i$ is an $(f+2s+i-1)$-extender for $r$, and a $(2,s)$\nobreakdash-spider $S$ rooted at vertex $r$ with $V(S) \cap F = \emptyset$, then $G$ contains a $(2,f+s)$\nobreakdash-spider rooted at $r$.
\end{lemma}
\begin{proof}
    We prove this lemma by induction on $f$.
    For $f=0$, the statement is clearly true.
    For $f>0$, observe that the vertex $x_1$ is an $(f+2s)$-extender for $r$, i.e.\ there are at least $f+2s$ elements in $O_{x_1,r}$.
    As $x_1 \in F$, $x_1 \notin O_{x_1,r}$, $r \in V(S)$, and $r \notin O_{x_1,r}$, we get that there are at most $(f-1 + 2s)$ elements in $(V(S)\cup F)\cap O_{x_1,r}$.
    Thus, there is an element $y \in O_{x_1,r}$ that is not in $V(S) \cup F$.
    We can construct a $(2,s+1)$\nobreakdash-spider $S'$ by adding the path $x_1 \to y \to r$ or the path $y \to x_1 \to r$ to the spider $S$.
    Let $F'=F\setminus \set{x_1}$.
    The result follows by applying the inductive hypothesis to $F'$ and $S'$.
\end{proof}
For every integer $i \ge 1$, note that an $(i+1)$-extender for $r$ is also an $i$-extender for $r$, and the requirement for $x_f$ in \cref{lem:greedy} is the strongest among all vertices in $F$. We can simplify, and weaken, the statement of \cref{lem:greedy} so that all extenders in $F$ are $(2f+2s-1)$-extenders.
\begin{corollary}\label{cor:greedy}
    For all integers $f \ge 0$, $s \ge 0$, $f+s=\ell$, a directed graph $G$, and a vertex $r \in V(G)$, if $G$ contains a set $F$ of $(2\ell-1)$-extenders for $r$ with $\norm{F}=f$, and a $(2,s)$\nobreakdash-spider $S$ rooted at vertex $r$ with $V(S) \cap F = \emptyset$, then $G$ contains a $(2,\ell)$\nobreakdash-spider rooted at $r$.
\end{corollary}
For a fixed integer $\ell$, directed graph $G$, and vertex $r$, we call a $(2\ell-1)$-extender for $r$ to be a \emph{strong extender} for $r$.
Observe that for a vertex $x$ with $x\to r$, and $\delta^-(x) \ge 2\ell$, there are at least $2\ell-1$ in-neighbors of $x$ that are different than $r$, and we get that $x$ is a strong extender for $r$.

We mention that in~\cite{hons2025unavoidable}, the authors used a similar technique of extending a smaller spider using vertices with in-degree at least $2\ell$ that are in-neighbors of the root vertex.
They divided the vertices of the directed graph into two sets: set $A$ of vertices with in-degree at least $2\ell$, and $B=V\setminus A$.
In their work, they used this technique to give an upper bound on the number of edges in $\pathset{AA}$.
This allowed for a lower bound on the number of paths in $\pathset{VBA}$ and the vertex $r$ with the most $\pathset{VBA}$ paths ending at $r$ was shown to be a root of a $(2,\ell)$\nobreakdash-spider.
See the proof of Theorem~1.8 in~\cite{hons2025unavoidable}.

Our approach is to use different kinds of extenders and apply \cref{cor:greedy} as the last step in the construction.
We also use a similar partition of the vertex set of $G$, but we select the root of the spider a little more carefully.
Then, we use edge coloring in an auxiliary graph to find a large spider that can be later extended to a $(2,\ell)$\nobreakdash-spider.

\begin{proof}[Proof of~\cref{thm:main}]
    Fix $\ell \ge 1$ and $d=2\ell$.
    Let $G$ be a directed graph with $\delta^+(v) = d$ for every vertex $v\in V(G)$.
    We will show that $G$ contains a $(2,\ell)$\nobreakdash-spider.
    The theorem will follow, as every directed graph with minimum out-degree at least $d$ contains a subgraph in which every vertex has out-degree exactly~$d$.

    Let $V=V(G)$ and $A\subseteq V$ denote the set of vertices of $G$ with in-degree at least $2\ell$.
    Let $B=V \setminus A$ denote the set of vertices with in-degree at most $2\ell-1$.
    As already observed in~\cite{hons2025unavoidable}, we can easily get the following bound on~$\norm{\pathset{VBA}}$:
    \begin{align*}
        \norm{\pathset{VBA}}
        &= \norm{\pathset{VBV}} - \norm{\pathset{VBB}} = \norm{\pathset{ABV}} + \norm{\pathset{BBV}} - \norm{\pathset{VBB}} \\
        &= \norm{\pathset{AVV}} - \norm{\pathset{AAV}} + \norm{\pathset{BBV}} - \norm{\pathset{VBB}} \\
        &\geq \norm{\pathset{AVV}} - \norm{\pathset{AAV}} + \norm{\pathset{BB}}\cdot(d-1 - (2\ell-1)) \\
        &\geq \norm{\pathset{AVV}} - \norm{\pathset{AAV}} \geq d(d-1)\norm{A} - d\norm{\pathset{AA}}\text{.}
    \end{align*}

    For any vertex $x \in A$, let $A_x \subseteq A$ be the set of in-neighbors of $x$ in $A$.
    Let $VB_x \subseteq \pathset{VBA}$ be the set of all simple $2$-paths with the second vertex in $B$ and the third vertex $x$.
    Select vertex $r \in A$ to be a vertex that maximizes the value $d\norm{A_r}+\norm{VB_r}$.
    By averaging over all vertices in $A$, we get:
    \[
        d\norm{A_r}+\norm{VB_r}
        \ge \frac{d\norm{\pathset{AA}}+\norm{\pathset{VBA}}}{\norm{A}}  
        \ge \frac{d\norm{\pathset{AA}} + d(d-1)\norm{A} - d\norm{\pathset{AA}}}{\norm{A}}
        \ge d^2 - d\text{.}
    \]

    Observe that every vertex $x \in A_r$ has $x \to r$ and $\delta^-(x)\ge 2\ell$.
    Thus, every $x \in A_r$ is a strong extender for $r$.
    Let $C_r$ be the set of all other strong extenders for $r$.
    Let $a=\norm{A_r}$ and $c=\norm{C_r}$. 
    We have $A_r \cap C_r = \emptyset$, and $A_r \cup C_r$ is a set of $a+c$ strong extenders for $r$.
    In order to apply \cref{cor:greedy} it remains to find a $(2,\ell-a-c)$\nobreakdash-spider rooted in $r$ that is disjoint with $A_r \cup C_r$.

    Let $Q_r$ be the set of paths in $VB_r$ that are disjoint with $A_r \cup C_r$.
    Each vertex $x$ in $A_r$ has $x \to r$ and $x \notin B$.
    Thus, there are no paths in $VB_r$ that have $x$ as the second vertex, and there are at most $d-1$ paths in $VB_r$ that have $x$ as the first vertex.
    As $d = 2\ell$, we have that there are at most $2\ell-1$ paths in $VB_r$ that have $x\in A_r$ as a vertex.
    For each vertex $x$ in $C_r$ there are at most $d$ paths in $VB_r$ that have $x$ as the first vertex.
    If there are paths in $VB_r$ that have $x$ as the second vertex, then $x \in B$, $\delta^-(x) < 2\ell$, and there are at most $2\ell-1$ such paths. 
    As $d = 2\ell$, we have that there are at most $d+2\ell-1=4\ell-1$ paths in $VB_r$ that have $x\in C_r$ as a vertex.
    Based on these observations, the previous bound on $\norm{VB_r}$, and $\norm{A_r}=a$ we get:
    \[
        \norm{Q_r} \ge \norm{VB_r} - a(2\ell-1) - c(4\ell-1) \ge d^2 -d - d\norm{A_r} - a(2\ell-1) - c(4\ell-1) \ge d^2 - d - (a+c)(4\ell-1)\text{.}
    \]

    Now, we define an undirected graph $H$ with $V(H) = V(G)\setminus (A_r \cup C_r \cup \set{r})$.
    We add an edge $\set{x,y}$ to $H$ if and only if there is a path $x \to y \to r$ or a path $y \to x \to r$ in $Q_r$.
    Observe that each undirected edge in $H$ corresponds to at least one directed edge in $G$ that can be extended to a simple $2$-path that ends in $r$. 
    Consider any vertex $x \in V(H)$.
    As vertices in $A_r \cup C_r$ are explicitly removed from $H$, we have that $x$ is not a strong extender for $r$.
    As every neighbor of $x$ in $H$ is an element of $O_{x,r}$, we get that the degree of $x$ in $H$ is at most $2\ell-2$.
    By Vizing's Theorem~\cite{Vizing1964}, there is a coloring of the edges of $H$ using $2\ell-1$ colors such that no two incident edges share a color.
    Let $s$ be the largest size of a color class in the coloring, and $T_r$ be some color class of size $s$.
    We have $s \ge \norm{Q_r}/(2\ell-1)$.
    As edges in $T_r$ are pairwise disjoint, we can extend each edge in $T_r$ to a simple $2$-path that ends in $r$ and get a $(2,s)$\nobreakdash-spider $S_r$ rooted at $r$.
    We easily get the following
    \[
        s(2\ell-1) \ge \norm{Q_r} \ge d^2 -d -(a+c)(4\ell-1)\text{,}
    \]
    and
    \[
        (a+c+s)(4\ell-1) \ge d^2 - d = 4\ell^2 - 2\ell\text{,}
    \]
    and
    \[
        (a+c+s) \ge \ell\text{.}
    \]

    To finish the proof, observe that if $s \ge \ell$, then $G$ already contains a $(2,\ell)$\nobreakdash-spider as a subgraph.
    Otherwise, we have $1 \le \ell-s \le a+c$, and we can select a set of strong extenders $F\subseteq A_r\cup C_r$ with exactly $\ell-s$ elements.
    Finally, by \cref{cor:greedy} applied to $F$ and $S_r$, we get that $G$ contains a $(2,\ell)$\nobreakdash-spider as a subgraph.
\end{proof}
 
\section{An Almost-Linear-Time Algorithm}

It is worth mentioning that the presented proof of \cref{thm:main} yields a simple, almost-linear-time algorithm that finds a $(2,\ell)$\nobreakdash-spider in any directed graph $G$ satisfying the conditions of the theorem.
Let $n=\norm{V(G)}$, $m=\norm{E(G)}$, and observe that selecting a subgraph of $G$ with out-degree of every vertex equal to $d$ is easily done in $\Oh{n\ell}$ time.
Further, calculating $d\norm{A_x}+\norm{VB_x}$ for every vertex $x \in A$ can be done in $\Oh{n\ell}$ time.
This allows to select the root~$r$ of the spider in linear time.
The sets $A_r$ and $C_r$ are easily constructed in linear time.
Recent breakthroughs in edge coloring algorithms~\cite{assadi2025,assadi2026} allow us to construct the edge coloring of $H$ in $m^{1+\oh{1}}$ time.
The final application of \cref{lem:greedy} is also easily implemented in time $\Oh{\ell^2}$.
Thus, the total time used for the construction of a $(2,\ell)$\nobreakdash-spider is almost-linear in the size of $E(G)$.

\section{Future Directions}

Although \cref{thm:main} establishes the optimal minimum out-degree threshold that forces the existence of a $(2,\ell)$-spider in directed graphs, several natural questions remain open. One direction is to classify the extremal examples that achieve this threshold while avoiding a $(2,\ell)$-spider, and to understand the structural features that make such constructions possible. A second direction is to extend the result to general $(k,\ell)$-spiders and prove the Giant Spider Conjecture for directed graphs. 
It seems plausible that \cref{thm:main} can be strengthened in the setting of oriented graphs by proving that any oriented graph with minimum out-degree at least $\sqrt{2}\cdot\ell$ contains a $(2,\ell)$-spider, through a more refined application of similar methods.
However, achieving further improvements beyond this bound is likely to require fundamentally new techniques.

\bibliography{spiders}

\end{document}